\documentclass[a4paper,12pt]{article}

\usepackage{amsmath}
\usepackage{amsthm}
\usepackage{amsfonts}
\newtheorem{theorem}{Theorem}

\newtheorem{proposition}[theorem]{Proposition}

\newcommand{\pr}{\mathbb{P}}
\newcommand{\B}{\textbf{B}}

\newcommand{\R}{\mathbb{R}}

\newcommand{\Z}{\mathbb{Z}}

\title{The entirely coupled region of \\ supercritical contact processes} 

\date{\vspace{-6mm}}
\author{Achillefs Tzioufas\footnote{University}}
\begin{document}
\maketitle 

\begin{abstract} 
We consider translation-invariant, finite range, supercritical contact processes. We show the existence of unbounded space-time cones within which the descendancy of the process from full occupancy may with positive probability be identical to that of the process from the single site at its apex. The proof comprises an argument that leans upon refinements of a successful coupling among these two processes, and is valid in $d$-dimensions. 
\end{abstract} 

{\bf 1. Introduction.} The \textit{contact process} is an extensively studied class of spatial Markov process introduced by Harris \cite{H74} in 1974; contact distributions were considered first in Mollison \cite{M72}, for later developments in this regard, see also Mollison \cite{M77}. The process can be viewed as a simple model for spatial growth, or the spread of an infection in a spatially structured population. In this note we will adopt the perspective and associated terminology stemming from the former interpretation. We have opted to work our proofs in detail in one (spatial) dimension, since the $d$-dimensional extension is directly analogous and is omitted as such. We consider the following class of \textit{translation-invariant} and \textit{finite range} contact processes $(\xi_{t}: t \geq0)$. Regarding sites in $\xi_{t}$ as \textit{occupied} and others as \textit{vacant}, the process with parameters $\underline{\mu} = (\mu_{i}; \mbox{ } i = -M, \dots, M, i \not= 0)$ evolves according to the following local prescription: 

\noindent (i) Particles die at rate 1.

\noindent (ii) A particle at $x$ gives birth at rate $\mu_{y-x}$ at $y$, $|x-y|\leq M$, $y \not=x$.

\noindent (iii) There is no more than one particle per site\footnote{This may be thought of either as that 
births at occupied sites are disallowed, or that those births take place, but result in merging and coalescence with the already existent, at the site that birth is given, particle; the latter conceptualizing will be in force in the sequel.}.
\vspace{1mm}

\noindent Note that equivalently, the rate at which births occur can be any linear combination of the parameter rates, and that deaths may occur at any constant rate.  For background regarding this process we refer the reader to Liggett \cite{L85, L99} and Durrett \cite{D95}. 

Whenever $\{\xi_{t} \not= \emptyset \mbox{ } \forall t\}$ occurs, we say that the process \textit{survives}, or that \textit{survival}  occurs; whenever $\{\xi_{t} \not= \emptyset \mbox{ } \forall t\}^{c}$ occurs we say that the process  \textit{dies out}. 
We say that the process is \textit{supercritical} whenever  $\underline{\mu}$ is such that the probability of survival from a finite sets is strictly positive. Note that, although this is not the classical definition of supercriticality, which is, for instance in the uniform, symmetric interaction case, i.e.\ $\mu:= \mu_{i}$ for all $i$, that $\mu > \mu_{c} := \inf\{\mu: \pr(\xi_{t} \not= \emptyset \mbox{ } \forall t)>0\}$,  the two definitions are equivalent for it is known that the process for the parameters on the critical surface dies out with probability one, cf.\ Bezuidenhout and Grimmett \cite{BG90} and Bezuidenhout and Gray \cite{BG94}.  

Let $\xi_{t}^{O}$ be the process started from $\{0\}$ and further let $r_{t} = \sup\xi_{t}^{O}$ and $l_{t} = \inf\xi_{t}^{O}$ be respectively the rightmost and leftmost sites of $\xi_{t}^{O}$. In addition, let  $\alpha$  (resp.\ $\beta$) be the \textit{asymptotic velocity}\footnote{a consequence of Liggett's \cite{L85} version of Kingman's subadditive theorem \cite{K68}, originally derived in Durrett \cite{D80}.} of $r_{t}$ (resp.\ $l_{t}$), i.e.\ $\displaystyle{r_{t}/t \rightarrow \alpha}$ and $\displaystyle{l_{t}/t \rightarrow \beta}$, as $t \rightarrow \infty$, a.s. on survival.
Furthermore, let $I_{t} = [(\beta+\epsilon)t, (\alpha-\epsilon)t]$, $t \in \R_{+}$, where $\epsilon>0$ is such that $\beta+\epsilon < \alpha-\epsilon$. We can now state the result.

\begin{theorem}
Let $\xi_{t}^{O}$ and $\xi_{t}^{\Z}$ be the supercritical contact process started from $\{0\}$ and $\Z$ respectively. We have that: 
$\displaystyle{\xi^{\Z}_{t} \cap I_{t} =\xi^{O}_{t} \cap I_{t}, \mbox{ for all } t, }$
with strictly positive probability. 
\end{theorem}

Note that no loss of generality is incurred by the proviso on the asymptotic velocities in Theorem 1 that is known to be equivalent to supercriticality of the class of processes considered; cf.\ with Theorem 1, $\mathsection$ 3, Durrett and Schonmann \cite{DS87}. The proof goes through refining the argument used in showing that there is a positive chance for two contact processes started from all sites and from any finite set to agree on that set for all times\footnote{Note however that this result relies on a different assumption, namely that the two processes can be coupled to agree on finite sets for all sufficiently large times with probability one on the event that the finite process survives; this condition cannot be in general dropped for this to hold. To see this consider for instance the one-sided interaction process from finite sets (i.e.\ $\mu_{i} \not= 0$ for $i \geq 1$) 
for which the finite dimensional distributions converge to the Dirac measure on the empty set: that is, in the supercritical case, despite that the cardinality of particles tends to infinity when the process survives, 
particles wander off to infinity.}, employed in the proofs of the central limit theorem for the endmost particles of the process; cf.\ with Chpt.\ 4, Remark 2, Tzioufas \cite{T11}. The proof relies on a direct probabilistic argument and does not rely on renormalization arguments other than those necessitated in the proof of the so-called lower-inclusion part of the asymptotic shape theorem (see the explanatory comments below for term usage). 

The first formulation of a version of Theorem 1 in the literature is the so-called "formation of the descendancy barriers"; for precise statements and proofs cf.\ with $\mathsection$ 2.2, p.\ 913 ff., Andjel \textit{et.al.} \cite{AMPV10}. The special case of the analog of Theorem 1 for the so-called \textit{basic} case (i.e.\ nearest neighbors and symmetric interaction) in $d$-dimensions may also be derived by a simple Borel-Cantelli lemma application through Theorem 1.1 of Garet and Marchand \cite{GM14}, which offers large deviations estimates corresponding to the extension of the shape theorem for this process in a random environment achieved in Garet and Marchand \cite{GM12}  by building upon various large deviations estimates associated to the original shape theorem.

Theorem 1 will be derived as a consequence of Proposition 2 given below. We now comment on the intimate relation among them 
for the (supercritical) process in dimension one and under the simplifying assumption of basic interaction. The \textit{asymptotic shape theorem} in this case may be shown as a simple corollary of the almost sure existence of a strictly positive asymptotic velocity of the rightmost site. To see this recall the well-known fact (cf.\ Durrett \cite{D80}) that, by simple two-dimensional path intersection properties of the graphical representation (cf.\ Harris \cite{H78}), the process started from a singleton and that started from all sites are identical amongst the endmost sites of the former, provided that its descendancy is not empty. In the light of this coupling, the following statement that we refer to as \textit{the lower-inclusion} part of the asymptotic shape theorem is then immediate: the descendancy of the process from the single site at the origin and from full occupancy are identical within any linearly growing at rate smaller than the asymptotic velocities interval, for all sufficiently large times, almost surely provided survival occurs\footnote{the upper-inclusion counterpart being that the descendants of former process are contained within any linearly growing at rate larger than the asymptotic velocities interval, for all sufficiently large times, almost surely provided survival occurs.}. In view of these comments, we may now phrase the contents of Theorem 1 as follows. Provided that the random \textit{coupling time} associated to the lower inclusion part of the asymptotic shape theorem is almost surely finite, the corresponding \textit{coupling event} may commence from time zero with positive probability.

\vspace{2mm}

{\bf 2. Preliminaries. } The following result is derived in Durrett and Schonmann \cite{DS87} by means of their extension of the renormalized  construction of Durrett and Griffeath \cite{DG83} to translation invariant, finite range, discrete time one-dimensional contact processes (another outline of these arguments may also be found in $\mathsection$ 4 below).  Note that we shall use coordinatewise notation equally well when convenient, that is we write $\xi(\cdot) = 1(\cdot \in \xi)$, where $1(\cdot)$ denotes the indicator function.

\begin{proposition}[(2),  $\mathsection 6$ in \cite{DS87}] \label{shape} Grant assumptions in Theorem 1. Let $R_{t} = \sup_{s\leq t} r_{s}$  and  $L_{t} = \inf_{s\leq t} l_{s}$. Then $\displaystyle{\{x : \xi_{t}^{O}(x)  = \xi_{t}^{\Z}(x) \mbox{ and } L_{t} \leq x \leq R_{t} \}\supseteq I_{t} \cap \Z}$, for all $t$ sufficiently large, almost surely on $\{\xi_{t}^{O}\not= \emptyset,  \mbox{ for all } t\}$.
\end{proposition}

\begin{proof}[{\hspace{3mm} \bf 3. Proof of Theorem 1}]
Familiarity with the construction of $(\xi_{t}^{A}: A\subset \Z)$ by a graphical representation, a realization of which will be denoted typically by $\omega$, and standard associated terminology  is assumed, see for example the introductory sections in \cite{D95, L99}.  We will use the notation: for all $\omega \in E_{1}$, $\omega \in E_{2}$ a.e.\ to denote that $\pr(\{\omega: \omega \in E_{1}, \omega \not\in E_{2}\}) = 0$ (where a.e.\ stands for almost everywhere (on $E_{1}$)).

Let $A_{n} = \{\xi^{\Z}_{t} \cap I_{t} =\xi^{O}_{t} \cap I_{t}, \mbox{ for all } t \geq n \}$,
for integer $n\geq0$. Proposition $\ref{shape}$ implies that for all $\omega \in \{\xi_{t}^{O} \not= \emptyset, \mbox{ for all } t\}$, $\omega \in \{\xi^{O}_{t} \cap I_{t} = \xi^{\Z}_{t} \cap I_{t} \mbox{, for all }t\geq t_{0}\}$, for some $t_{0}$, a.e.. Hence $\pr\left(\cup_{n\geq 0} A_{n}\right)$ equals $\pr(\xi_{t}^{O} \not= \emptyset, \mbox{ for all } t)$ which is positive as the process is supercritical, which implies, say, by contradiction, that there is $n_{0}$ finite such that $\pr(A_{n_{0}}) >0$.  

We show that this last conclusion implies that, indeed, $\pr(A_{0})>0$. Let $A_{n_{0}}'$ be such that $\omega'\in A_{n_{0}}'$ if and only if there exists $\omega\in A_{n_{0}}$ such that $\omega$ and $\omega'$ are identical realizations except perhaps from any $\delta$-symbols (death marks) in $\B_{n_{0}} = V \times [0,n_{0}]$, where $V = \{-v_{l},\dots, v_{r}\}$ and $v_{r}, v_{l}$ are the smallest integer greater than $(\alpha- \epsilon) n_{0}$ and than $\min\{0, (\beta + \epsilon) n_{0}\}$ respectively. Let $\tilde{A}_{n} = \{\xi^{\Z}_{t} \cap I_{t} =\xi^{V}_{t} \cap I_{t}, \mbox{ for all } t \geq n \}$, and let $F$ be the event that no $\delta$-symbols appear in $\B_{n_{0}}$. We have that
\begin{eqnarray*}
\pr(A_{n_{0}}' \cap F ) &=& \pr(A_{n_{0}}') \pr(F) \nonumber \\
&\geq& \pr(A_{n_{0}}) e^{-(2|V|+1) n_{0}} >0
\end{eqnarray*}
and that $\tilde{A}_{0} \supseteq A_{n_{0}}' \cap F$, where to see that the last claim is true note that if $\omega$ and $\omega'$ are identical except that $\omega'$ does not contain any $\delta$-symbols in $\B_{n_{0}}$ that possibly exist in $\omega$, then $\omega \in A_{n_{0}}$ implies that\footnote{To infer this, consider sampling from $A_{n_{0}}' \cap F$ and note that adjusting the death marks by mapping on $F$ results in enlarging $\xi_{t}^{O}$ and further that, as paths outside of the box $\B_{n_{0}}$ are left intact by such adjustments, particles of $\xi_{t}^{\Z \backslash V}$ cannot go through $\xi_{t}^{O} \cap I_{t} \times t$, for all $t \in [n_{0}, \infty)$, via trajectories \textit{not} intersecting with the box (property inherited from $A_{n_{0}}$). Further, as $\B_{n_{0}} \subset (\xi_{t}^{V} \times t: t \leq n_{0})$, those particles cannot go through $\xi_{t}^{V} \cap I_{t} \times t$, for all $t \in [n_{0}, \infty)$, via paths intersecting with the box, because they are assimilated (by merging and coalescence) to $\xi_{t}^{V}$. Finally, they cannot go through $\xi_{t}^{V} \cap I_{t} \times t$, for $t \in [0, n_{0})$, simply because by definition of $V$, $I_{t}  \times t \subset \B_{n_{0}}$.} $\omega' \in \tilde{A}_{n_{0}}$, and indeed $\omega' \in \tilde{A}_{0}$. Hence $\pr(\tilde{A}_{0})>0$ and from that the result follows by the Markov property for $\xi_{t}^{O}$ at an appropriately chosen sufficiently small time. 
\end{proof}

\begin{proof}[\hspace{4mm}{\bf 4. Proof outline for Proposition 2.}] We follow arguments from $\mathsection$ 5, 6 in \cite{DS87} that regard the discrete-time version of the process, giving some additional remarks and referring to the original work for more details whenever necessary. Consider the graphical representation for a supercritical contact process which satisfies the assumptions given in Theorem 1.  
We will argue that for any $\epsilon>0$ there exist $C, \gamma \in (0,\infty)$ such that
\begin{equation}\label{fds}
\pr(\xi_{2t}^{O}(x) \not= \xi^{\Z}_{2t}(x), \xi_{2t}^{O} \not= \emptyset) \leq Ce^{-\gamma t}
\end{equation}
for all $x \in I_{2t}$ and $t\geq0$. To see that this suffices, note that the result for integer times then follows from (\ref{fds}) and the 1st Borel-Cantelli lemma immediately, since 
\[
\sum\limits_{n\geq1} \pr\big(\bigcup_{x\in I_{2n}} \xi_{2n}^{O}(x) \not= \xi^{\Z}_{2n}(x) | \hspace{1mm}  \xi_{t}^{O} \not= \emptyset, \mbox{ for all } t\big) < \infty, 
\]
where we first used that  $\pr(\xi_{t}^{O} \not= \emptyset \cap  \{\xi_{t}^{O} \not= \emptyset, \mbox{ for all } t\}^{c})$ is exponentially bounded in $t$, a result proved by emulating a standard version of a restart argument, see Theorem 2.30 (a) in \cite{L99}. Whereas, obtaining the result for all $t$ then follows elementarily by using a "filling in" argument, emulating for instance (4.2.2) in Chpt. 4 in \cite{T11}.

To show (\ref{fds}) one considers the dual process $(\tilde{\xi}_{s}^{x}; 0 \leq s < t)$  defined on the same graphical representation by reversal of arrows over the time interval $(t,2t]$. Note that this process is independent of $(\xi_{s}^{O}, 0 \leq s  \leq t)$ by independence of the Poisson process in disjoint parts of the representation and that its distribution w.r.t.\ the space-time point $x \times t$ is equal to a copy of the process w.r.t.\ the origin and parameters $\mu_{i} = \mu_{-i}$.  (Hence, observe that in particular, if $\tilde{l}_{s} = \inf\{y: y \in\tilde{\xi}_{s}^{x}\}$, $(\tilde{l}_{s}-x)$ is equal in distribution to $(- r_{s})$, and similarly for $\sup\{y: y \in\tilde{\xi}_{s}^{x}\}$). Then,  (\ref{fds})  follows by showing that, there exist $C, \gamma \in (0,\infty)$, such that
\begin{equation}\label{inters}
\pr(\xi_{t}^{O} \cap \tilde{\xi}_{t}^{x} = \emptyset, \xi_{t}^{O} \not= \emptyset,  \tilde{\xi}_{t}^{x} \not= \emptyset) \leq Ce^{-\gamma t},
\end{equation}
for all $x \in I_{2t}$, where, to see that this suffices, note that $\{\xi_{t}^{O}\cap \tilde{\xi}_{t}^{x} \not= \emptyset,  \xi_{t}^{O} \not= \emptyset,  \tilde{\xi}_{t}^{x} \not= \emptyset\}$ is contained in $\{\xi_{2t}^{O}(x)=\xi_{2t}^{\Z}(x)=1, \xi_{2t}^{O} \not=\emptyset\}$.

The proof of (\ref{inters}) then goes through imbedding the rescaled coupled with oriented site percolation construction in order to deduce this from the corresponding result for the last process with density arbitrarily close to 1.  To show this first, let $K_{n}$ be independent retaining probability $p_{S}<1$ oriented site percolation on the usual lattice $\overrightarrow{\mathbb{L}}$ with set of sites the space-time points $(y,n) \in \Z^{2} $ such that $y+n \mbox{ is even}$ and  $ n \geq0$ obtained by adding a bond from each such point $(y,n)$ to $(y-1,n+1)$ and to $(y+1,n+1)$.  Further, let $\tilde{K}_{n}$ be an independent copy of  $K_{n}$ and let also $X_{n} = \{x_{1}, \dots, x_{[cn]}\}$ and $\tilde{X}_{n}= \{\tilde{x}_{1}, \dots, \tilde{x}_{[cn]}\}$ be collections of points at level $n$ with spatial coordinate $y \leq a n $, $a<1$ , $c >0$. Furthermore, let $E_{k}$ denote the event that both $x_{k} \in K_{n}$ and $\tilde{x}_{k} \in \tilde{K}_{n}$, 
and also let $(E'_{k})$ denote an independent thinning of probability $p'>0$ of the $(E_{k})$,  $k=1, \dots [cn]$. We can now state the following easy consequence of Proposition 3 in \cite{TF} we need to employ. If $p_{S}$ is sufficiently close to $1$ then, for any $c>0$,  there are $q<1$ and $C<\infty$ such that 
\begin{equation}\label{op}
\pr\left(\cap_{k=1}^{[c n]} \bar{E}'_{k}, \{ K_{n} \not= \emptyset, \tilde{K}_{n}\not= \emptyset\}\right) \leq Cq^{n}, 
\end{equation}
for any $X_{n}$, $\tilde{X}_{n}$, and where $\bar{E}'_{k}$ denotes the complement of $E'_{k}$.  

Assuming that the comparison of  $(\xi_{s}^{O}, 0 \leq s < t)$ with $(K_{n})$ takes effect (successful embedding) by time 1, then the renormalized sites spread across the entire interval $I_{n T}$, where $n$ is such that $nT <t-1 <(n+1)T$  and $T$ is the corresponding renormalized time constant. Hence, if the comparison of  $(\tilde{\xi}_{s}^{x}, 0 \leq s < t)$ with $(\tilde{K}_{n})$ also takes effect after time 1, we have that, for any $x$,  there is a $c= c(\alpha,\beta)>0$ such that $[ct]$ renormalized sites of the two processes spatially overlap (to check this, take the observation following the definition $\tilde{\xi}_{s}^{x}$ into account and note that the rightmost (resp.\ leftmost) renormalized sites in a successful construction advance with asymptotic velocity $\alpha-\epsilon$ (resp.\ $\beta+\epsilon$)).   
We then derive (\ref{inters}) from $(\ref{op})$ by simply pairing up such spatially overlapping sites and choosing $p'>0$ there to be the probability of joining pairs by a path which is spatially constrained within the span of paired sites.
For this we use the independence of observing a joining among distinct pairs, which is due to the said constrain and independence of events measurable with respect to disjoint parts of the graphical representation. Finally, to work around the difficulty arising from an arbitrary starting point of the successful embeddings, one employs a "restart" technique that incorporates some basic geometrical considerations, see e.g.\ Proposition 2.8 in \cite{T11}. Following the arguments there, this difficulty may be overcome by showing that, outside of an event of exponentially small probability in $t$, one is reduced to working on the event for which the said conclusion regarding the order of spatially overlapping renormalized sites for the two processes again holds, and the last argument applies again to finish the proof. 

\end{proof}

\end{document}